\documentclass[10pt,twoside,draft]{article}

\usepackage{amsmath,amscd}
\usepackage{amssymb}
\usepackage{amsthm}
\usepackage{mathrsfs}
\usepackage{fixme}
\usepackage{enumerate}
\usepackage{graphicx}                    
\usepackage{pgf}
\usepackage{tikz}
\usetikzlibrary{arrows}
\usetikzlibrary{positioning}
\usetikzlibrary{shapes.geometric}
\usetikzlibrary{calc}

\def\YEAR{\year}\newcount\VOL\VOL=\YEAR\advance\VOL by-1995
\def\firstpage{1}\def\lastpage{1000}
\def\received{}\def\revised{}
\def\communicated{}

\makeatletter
\def\magnification{\afterassignment\m@g\count@}
\def\m@g{\mag=\count@\hsize6.5truein\vsize8.9truein\dimen\footins8truein}
\makeatother

\font\eightrm=cmr8
\font\caps=cmcsc10                    
\font\Caps=cmcsc10 scaled \magstep1   

\def\bfseries{\normalsize\caps}

\makeatletter
\@specialpagefalse\headheight=8.5pt
\def\DocMath{}
\renewcommand{\@evenhead}{%
    \ifnum\thepage>\lastpage\rlap{\thepage}\hfill%
    \else\rlap{\thepage}\slshape\leftmark\hfill{\caps\SAuthor}\hfill\fi}%
\renewcommand{\@oddhead}{%
    \ifnum\thepage=\firstpage{\DocMath\hfill\llap{\thepage}}%
    \else{\slshape\rightmark}\hfill{\caps\STitle}\hfill\llap{\thepage}\fi}%
\makeatother

\def\TSkip{\bigskip}
\newbox\TheTitle{\obeylines\gdef\GetTitle #1
\ShortTitle  #2
\SubTitle    #3
\Author      #4
\ShortAuthor #5
\EndTitle
{\setbox\TheTitle=\vbox{\baselineskip=20pt\let\par=\cr\obeylines%
\halign{\centerline{\Caps##}\cr\noalign{\medskip}\cr#1\cr}}%
	\copy\TheTitle\TSkip\TSkip%
\def\next{#2}\ifx\next\empty\gdef\STitle{#1}\else\gdef\STitle{#2}\fi%
\def\next{#3}\ifx\next\empty%
    \else\setbox\TheTitle=\vbox{\baselineskip=20pt\let\par=\cr\obeylines%
    \halign{\centerline{\caps##} #3\cr}}\copy\TheTitle\TSkip\TSkip\fi%
\centerline{\caps #4}\TSkip\TSkip%
\def\next{#5}\ifx\next\empty\gdef\SAuthor{#4}\else\gdef\SAuthor{#5}\fi%
\ifx\received\empty\relax
    \else\centerline{\eightrm Received: \received}\fi%
\ifx\revised\empty\TSkip%
    \else\centerline{\eightrm Revised: \revised}\TSkip\fi%
\ifx\communicated\empty\relax
    \else\centerline{\eightrm Communicated by \communicated}\fi\TSkip\TSkip%
\catcode'015=5}}\def\Title{\obeylines\GetTitle}
\def\Abstract{\begingroup\narrower
    \parskip=\medskipamount\parindent=0pt{\caps Abstract. }}
\def\EndAbstract{\par\endgroup\TSkip}

\long\def\MSC#1\EndMSC{\def\arg{#1}\ifx\arg\empty\relax\else
     {\par\narrower\noindent%
     2000 Mathematics Subject Classification: #1\par}\fi}

\long\def\KEY#1\EndKEY{\def\arg{#1}\ifx\arg\empty\relax\else
	{\par\narrower\noindent Keywords and Phrases: #1\par}\fi\TSkip}

\newbox\TheAdd\def\Addresses{\vfill\copy\TheAdd\vfill
    \ifodd\number\lastpage\vfill\eject\phantom{.}\vfill\eject\fi}
{\obeylines\gdef\GetAddress #1
\Address #2 
\Address #3
\Address #4
\EndAddress
{\def\xs{4.3truecm}\parindent=0pt
\setbox0=\vtop{{\obeylines\hsize=\xs#1\par}}\def\next{#2}
\ifx\next\empty 
     \setbox\TheAdd=\hbox to\hsize{\hfill\copy0\hfill}
\else\setbox1=\vtop{{\obeylines\hsize=\xs#2\par}}\def\next{#3}
\ifx\next\empty 
     \setbox\TheAdd=\hbox to\hsize{\hfill\copy0\hfill\copy1\hfill}
\else\setbox2=\vtop{{\obeylines\hsize=\xs#3\par}}\def\next{#4}
\ifx\next\empty\ 
     \setbox\TheAdd=\vtop{\hbox to\hsize{\hfill\copy0\hfill\copy1\hfill}
                \vskip20pt\hbox to\hsize{\hfill\copy2\hfill}}
\else\setbox3=\vtop{{\obeylines\hsize=\xs#4\par}}
     \setbox\TheAdd=\vtop{\hbox to\hsize{\hfill\copy0\hfill\copy1\hfill}
	        \vskip20pt\hbox to\hsize{\hfill\copy2\hfill\copy3\hfill}}
\fi\fi\fi\catcode'015=5}}\gdef\Address{\obeylines\GetAddress}

\newtheorem{thm}{Theorem}[section]
\newtheorem{definition}[thm]{Definition}
\newtheorem{lem}[thm]{Lemma}
\newtheorem{cor}[thm]{Corollary}

\newtheorem{prop}[thm]{Proposition}

\newtheoremstyle{example}
     {3 pt}
     {3 pt}
     {}          
     {}          
     {\bfseries} 
     {.}          
     {.5em}       
     {
     }
\theoremstyle{example}
\newtheorem{example}[thm]{Example}
\newtheorem{remark}[thm]{Remark}

\newtheorem*{cond}{Condition}

\DeclareMathOperator{\Id}{Id}

\newcommand{\N}{\mathbb{N}}
\newcommand{\Z}{\mathbb{Z}}

\newcommand{\subsubsubsection}[1]{\textsc{#1}.}

\renewcommand{\AA}{\mathcal{A}}
\newcommand{\BB}{\mathcal{B}}
\newcommand{\CC}{\mathcal{C}}
\newcommand{\DD}{\mathcal{D}}
\newcommand{\FF}{\mathcal{F}}
\newcommand{\LL}{\mathcal{L}}
\newcommand{\OO}{\mathcal{O}}
\newcommand{\PP}{\mathcal{P}}
\newcommand{\EE}{\mathcal{E}}
\newcommand{\GG}{\mathcal{G}}
\newcommand{\HH}{\mathcal{H}}

\newcommand{\X}{\mathsf{X}}

\newcommand{\psc}{W}

\renewcommand{\frac}[2]{#1 / #2}

\newcommand{\class}[1]{[#1]}

\newcommand{\Cs}{C^\ast}

\begin{document}
\Title
{On the Structure of Covers of Sofic Shifts}
\ShortTitle 
\SubTitle   
\Author 
{Rune Johansen}
\ShortAuthor 
\EndTitle
\Abstract 
A canonical cover generalizing the left Fischer cover to arbitrary sofic
shifts is introduced and used to prove that the left Krieger cover
and the past set cover of a sofic shift can be divided into
natural layers.
These results are used to
find the range of a flow-invariant
and to
investigate the ideal structure of the
universal $\Cs$-algebra associated to a sofic shift space. 
\EndAbstract
\MSC 
{Primary: 37B10; Secondary: 46L05} 
\EndMSC
\KEY 
{shift space, sofic shift, Krieger cover, Fischer cover}     
\EndKEY
\Address 
Rune Johansen
Department of \\ Mathematical Sciences
University of Copenhagen
Universitetsparken 5
2100 K\o benhavn \O{}
Denmark
\Address
\Address
\Address
\EndAddress

\section{Introduction}
\label{sec_introduction} 
Shifts of finite type have been completely classified up to flow
equi\-valence by Boyle and Huang \cite{boyle_huang}, but very little is
known about the classification of the class of sofic shift spaces
introduced by Weiss \cite{weiss} even though they are a natural first 
generalization of shifts of finite type. 
The purpose of this paper is to investigate the structure of -
and relationships between - various standard presentations (the Fischer
cover, the Krieger cover, and the past set cover) of
sofic shift spaces.
These results are used to find the range of the
flow-invariant introduced in \cite{bates_eilers_pask}, and to 
investigate the ideal structure of the $\Cs$-algebras associated to
sofic shifts.
In this way, the present paper can be seen as a continuation of
the strategy applied in
\cite{carlsen_eilers_ergod_2004,carlsen_eilers_doc_2004,matsumoto_2001},
where invariants for shift spaces are extracted from the associated
$\Cs$-algebras.  

Section \ref{sec_background} recalls the definitions of shift spaces,
labelled graphs, and covers to make the paper self contained.
Section \ref{sec_gfc} introduces a canonical and flow-invariant cover generalizing the left Fischer cover to arbitrary sofic shifts.

Section \ref{sec_foundation} introduces the concept of a foundation of a cover, which is used to prove that the left Krieger cover and the past set cover can be divided into natural layers and to show that the left Krieger cover of an arbitrary sofic shift can be identified with a subgraph of the past set cover.

In Section \ref{sec_invariant}, the structure of the layers of the left Krieger cover of an irreducible sofic shift is used to find the range of the flow-invariant introduced in \cite{bates_eilers_pask}. 
Section \ref{sec_cs} uses the results about the structure of covers of sofic shifts to investigate ideal lattices of the associated $\Cs$-algebras. Additionally, it is proved that 
Condition $(\ast)$ introduced by Carlsen and Matsumoto \cite{carlsen_matsumoto} holds if and
only if the left Krieger cover is the maximal essential subgraph of
the past set cover.

\subsubsubsection{Acknowledgements}
This work was supported by the Danish National Research Foundation (DNRF) through the Centre for Symmetry and Deformation.
The author would like to thank David Pask, Toke Meier Carlsen, and S\o
ren Eilers for interesting discussions and helpful comments. The author would also like to thank the anonymous referee for useful comments improving the exposition and to thank the University of Wollongong,
Australia and the University of Tokyo, Japan where parts of the
research for this paper were carried out during visits funded by
\emph{Rejselegat for Matematikere}. 

\section{Background}
\label{sec_background}
\subsubsubsection{Shift spaces}
Here, a short introduction to the definition and properties of shift
spaces is given to make the present paper self-contained; 
for a thorough treatment of shift spaces see \cite{lind_marcus}. 
Let $\AA$ be a finite set with the discrete topology. The
\emph{full shift} over $\AA$ consists of the space $\AA^\Z$ endowed
with the product topology and the \emph{shift map} $\sigma \colon
\AA^\Z \to \AA^\Z$ defined  by $\sigma(x)_i = x_{i+1}$ for all $i \in
\Z$. Let $\AA^*$ be the collection of finite words (also known as
blocks) over $\AA$. A
subset $X \subseteq \AA^\Z$ is called a \emph{shift space} if it
is invariant under the shift map and closed. 
For each $\FF \subseteq \AA^*$, define  $\X_\FF$ to be the set of
bi-infinite sequences in $\AA^\Z$ which do not contain any of the
\emph{forbidden words} from $\FF$. 
A subset $X \subseteq \AA^\Z$ is a shift space if and only if there
exists $\FF \subseteq \AA^*$ such that $X = \X_\FF$
(cf. \cite[Proposition 1.3.4]{lind_marcus}). $X$ is said to be 
a \emph{shift of finite type} (SFT) if this is possible for a finite
set $\FF$.

The \emph{language} of a shift space $X$ is defined to be the set of all words which occur in at least one $x \in X$, and it is denoted $\BB(X)$.
$X$ is said to be
\emph{irreducible} if there for every $u,w \in \BB(X)$ exists $v \in
\BB(X)$ such that $uvw \in \BB(X)$. 
For each $x \in X$, define the \emph{left-ray} of $x$ to be $x^- = \cdots
x_{-2} x_{-1}$ and define the \emph{right-ray} of $x$ to be $x^+ = x_0
x_1 x_2 \cdots$. The sets of all left-rays and all right-rays are,
respectively, denoted  $X^-$ and $X^+$.

A bijective, continuous, and shift commuting map between two shift
spaces is called a \emph{conjugacy}, and when such a map exists, 
the two shift spaces are said to be \emph{conjugate}. \emph{Flow
  equivalence} is a weaker equivalence relation generated by conjugacy
and \emph{symbol expansion} \cite{parry_sullivan}.

\subsubsubsection{Graphs}
For countable sets $E^0$ and $E^1$, and maps $r,s \colon E^1 \to E^0$
the quadruple $E = (E^0,E^1,r,s)$ is called a \emph{directed graph}. The
elements of $E^0$ and $E^1$ are, respectively, the vertices and the
edges of the graph.
For each edge $e \in E^1$, $s(e)$ is the vertex where $e$ starts,
and $r(e)$ is the vertex where $e$ ends. A \emph{path} $\lambda = e_1 \cdots
e_n$ is a sequence of edges such that $r(e_i) = s(e_{i+1})$ for all $i
\in \{1, \ldots n-1 \}$.
For each $n \in \N_0$, the set of paths of length $n$ is denoted
$E^n$, and the set of all finite paths is denoted $E^*$.
Extend
the maps $r$ and $s$ to $E^*$ by defining $s(e_1 \cdots e_n) = s(e_1)$
and $r(e_1 \cdots e_n) = r(e_n)$. A \emph{circuit} is a path $\lambda$
with $r(\lambda) = s(\lambda)$ and $\lvert \lambda \rvert > 0$.
For $u,v \in E^0$, $u$ is said to be \emph{connected} to $v$ if there is a path $\lambda \in E^*$ such that $s(\lambda) = u$ and $r(\lambda) = v$, and this is denoted by $u \geq v$ \cite[Section 4.4]{lind_marcus}.
A vertex is said to be \emph{maximal}, if it is connected to all other
vertices. $E$ is said to be irreducible if all vertices are maximal.
If $E$ has a unique maximal vertex, this vertex is said to be the \emph{root} of $E$.
$E$ is said to be \emph{essential} if every vertex emits
and receives an edge.
For a finite essential directed graph $E$, the 
\emph{edge shift} $(\X_E, \sigma_E)$ is defined by
\begin{displaymath}
  \X_E = \left\{ x \in (E^1)^\Z \mid r(x_i) = s(x_{i+1}) \textrm{ for all }
  i \in \Z \right\}.
\end{displaymath}

A \emph{labelled graph} $(E, \LL)$ over an alphabet $\AA$ consists
of a directed graph $E$ and a surjective labelling map $\LL \colon E^1
\to \AA$. Extend the labelling map to $\LL \colon E^* \to \AA^*$ by
defining $\LL(e_1 
\cdots e_n) = \LL(e_1) \cdots \LL(e_n) \in \AA^*$.
For a finite essential labelled graph $(E, \LL)$, define the shift space $(\X_{(E, \LL)}, \sigma)$ by 
\begin{displaymath}
  \X_{(E, \LL)} = \left\{ \left( \LL(x_i) \right)_i \in \AA^\Z \mid 
                          x \in \X_E  \right\}.
\end{displaymath}
The
labelled graph $(E, \LL)$ is said to be a \emph{presentation} of the
shift space $\X_{(E, \LL)}$, and a \emph{representative} of a word $w \in
\BB(\X_{(E, \LL)}) $ is a path $\lambda \in E^*$ such that
$\LL(\lambda) = w$. Representatives of rays are defined analogously.
If $H \subseteq E^0$ then the \emph{subgraph of} $(E, \LL)$ \emph{induced by} $H$ is the labelled subgraph of $(E, \LL)$ with vertices $H$ and edges $\{ e \in E^1 \mid s(e),r(e) \in H \}$. 

\subsubsubsection{Sofic shifts}
A function $\pi \colon X_1 \to X_2$ between shift spaces $X_1$ and
$X_2$ is said to be a \emph{factor map} if it is continuous,
surjective, and shift commuting. A shift space is
called \emph{sofic} \cite{weiss} if it is the image of an SFT under a factor
map.
A shift space is sofic if and only 
if it can be presented by a finite labelled graph \cite{fischer}. 
A sofic shift
space is irreducible if and only if it can be presented by an
irreducible labelled graph (see \cite[Section 3.1]{lind_marcus}).
Let $(E, \LL)$ be a finite labelled graph and let $\pi_\LL \colon \X_E \to \X_{(E,  \LL)}$ be the factor map induced by the labelling map $\LL \colon
E^1 \to \AA$ then the SFT $\X_E$ is called a \emph{cover} of the
sofic shift $\X_{(E, \LL)}$, and $\pi_\LL$ is called the covering map. 

A presentation $(E,\LL)$ of a sofic shift space $X$ is said to be \emph{left-resolving}
if no vertex in $E^0$ receives two edges with the same label. 
Fischer proved \cite{fischer} that, up to labelled graph
isomorphism, every irreducible sofic shift has a unique left-resolving
presentation with fewer vertices than any other left-resolving
presentation. This is called the \emph{left Fischer cover} of $X$,
and it is denoted $(F, \LL_F)$. 
An irreducible sofic shift is said to have \emph{almost finite type} (AFT) \cite{marcus,nasu_aft} if the left Fischer cover is right-closing (see e.g. \cite[Def. 5.1.4]{lind_marcus}).

For $x \in \BB(X) \cup X^+$, define the \emph{predecessor set} of $x$ to be
the set of left-rays which may precede $x$ in $X$
(see \cite[Sections I and III]{jonoska_marcus} and \cite[Exercise  3.2.8]{lind_marcus}). The \emph{follower set} of a left-ray or word is defined
analogously.
Let $(E, \LL)$ be a labelled graph presenting $X$ and let  $v \in E^0$. Define the \emph{predecessor set} of $v$ to be the set of left-rays in $X$ which have a presentation terminating at $v$. This is denoted $P_\infty^E(v)$, or just $P_\infty(v)$ when $(E, \LL)$ is understood from the context. The presentation $(E, \LL)$ is said to be \emph{predecessor-separated} if $P_\infty^E(u) \neq P_\infty^E(v)$ when $u,v \in E^0$ and $u \neq v$.

The \emph{left Krieger cover} of the shift space $X$ is the labelled graph
$(K, \LL_K)$ where $K^0 = \{ P_\infty(x^+) \mid x^+ \in X^+\}$,
and where there is an edge labelled $a \in \AA$ from $P \in K^0$ to
$P' \in K^0$ if and only if there exists $x^+ \in X^+$ such that $P
= P_\infty(a x^+)$ and $P' = P_\infty(x^+)$.
The \emph{past set cover} of the shift space $X$ is the labelled graph
$(\psc, \LL_\psc)$ where $\psc^0 = \{ P_\infty(w) \mid w \in \BB(X) \}$ and where the edges and labels are constructed as in the Krieger cover.
A shift space is sofic if and only if the number of predecessor
  sets is finite \cite[\S 2]{krieger_sofic_I}, so the left Krieger
  cover is finite exactly when the shift space is sofic. 
The left Fischer cover, the left Krieger cover, and the past set cover are left-resolving and predecessor-separated presentations of $X$.

The right Krieger cover and the future set cover are right-resolving and follower-separated covers defined analogously to the left Krieger cover and the past set cover, respectively.
Every result developed for left-resolving covers in the
following has an analogue for the corresponding right-resolving
cover. These results can easily be obtained by considering the 
transposed shift space $X^\textrm{T}$ (see e.g. \cite[p. 39]{lind_marcus}).

\section{Generalizing the Fischer cover}
\label{sec_gfc} 
Jonoska \cite{jonoska} proved that a reducible sofic shift does not necessarily have a unique minimal left-resolving presentation. The aim of this section is to define a generalization of the left Fischer cover as the subgraph of the left Krieger cover induced by a certain subset of vertices. 
Let $X$ be a sofic shift space, and let $(K, \LL_K)$ be the
left Krieger cover of $X$. 
A predecessor set $P  \in K^0$ is said to be
\emph{non-decomposable} if $V \subseteq K^0$ and $P = \bigcup_{Q \in V} Q$ implies that $P \in V$.

\begin{lem}
\label{lem_decomposable}
If $P \in K^0$ is non-decomposable then the  subgraph of $(K, \LL_K)$ induced by $K^0 \setminus \{ P \}$ is not a presentation of $X$.  
\end{lem}

\begin{proof}
Let $E$ be the
subgraph of $K$ induced by $K^0 \setminus \{P\}$.
Choose $x^+ \in X^+$ such that $P = P_\infty(x^+)$.
Let $V \subseteq K^0 \setminus\{ P \}$ be the set of vertices where a
presentation of $x^+$ can start.
Then $Q \subseteq P_\infty(x^+) = P$ for each $Q \in V$, and by assumption, there exists $y^- \in P \setminus \bigcup_{Q \in V} Q$. Hence, there is no presentation of $y^-x^+$ in $(E, \LL_K|_E)$.       
\end{proof}

\noindent
Lemma \ref{lem_decomposable} shows that a subgraph of the left Krieger cover which presents the same shift must contain all the non-decomposable vertices. The next example shows that this subgraph is not always large enough.

\begin{example}
\label{ex_gfc_justifying}
It is easy to check that the labelled graph in Figure
\ref{fig_gfc_justifying} is the left Krieger cover of a reducible
sofic shift $X$. Note that the predecessor set $P$ is decomposable since $P = P_1
\cup P_2$, and that the graph obtained by removing the vertex $P$ and
all edges starting at or terminating at $P$ is not a presentation of
the same sofic shift since there is no presentation of $f^\infty
dbjk^\infty$ in this graph. Note, that there is a path from
$P$ to the vertex $P'$ which is non-decomposable. 
\end{example}

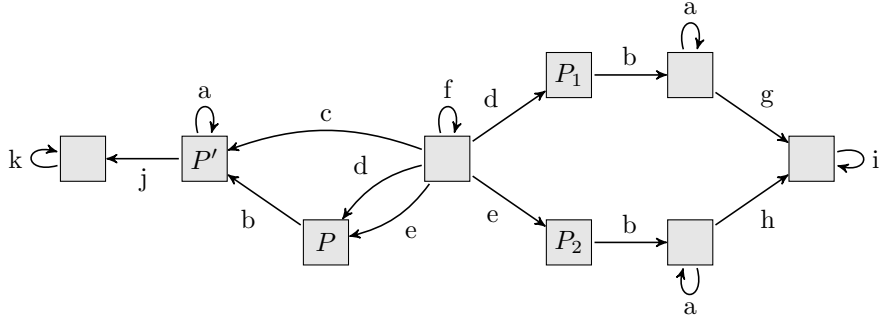
\begin{figure} 
\begin{center}
\begin{tikzpicture}
  [bend angle=45,
   knude/.style = {circle, inner sep = 0pt},
   vertex/.style = {circle, draw, minimum size = 1 mm, inner sep =
      0pt, fill=black!10},
   textVertex/.style = {rectangle, draw, minimum size = 6 mm, inner sep =
      1pt, fill=black!10},
   to/.style = {->, shorten <= 1 pt, >=stealth', semithick}]
  
  \matrix[row sep=5mm, column sep=10mm]{
   & & & & \node[textVertex] (P1) {$P_1$}; & \node[textVertex] (h1)
    {}; & \\
  \node[textVertex] (v2) {}; & \node[textVertex] (v1) {$P'$};  & & 
  \node[textVertex] (O) {}; & & & \node[textVertex] (h3) {}; \\
  &  & \node[textVertex] (P) {$P$}; & & \node[textVertex] (P2) {$P_2$};
  & \node[textVertex] (h2) {}; & \\
  };
  \draw[to, loop above]  (O) to node[auto] {f} (O);
  \draw[to, bend right=20] (O) to node[auto,swap] {c} (v1);
  \draw[to] (O) to node[auto] {d} (P1);
  \draw[to] (O) to node[auto,swap] {e} (P2);
  \draw[to] (P1) to node[auto] {b} (h1);
  \draw[to] (P2) to node[auto] {b} (h2);
  \draw[to, loop above]  (h1) to node[auto] {a} (h1);
  \draw[to, loop below]  (h2) to node[auto] {a} (h2);
  \draw[to] (h1) to node[auto] {g} (h3);
  \draw[to] (h2) to node[auto,swap] {h} (h3);
  \draw[to, loop right]  (h3) to node[auto] {i} (h3);
 
  \draw[to, bend right=20] (O) to node[auto,swap] {d} (P);
  \draw[to, bend left=20] (O) to node[auto] {e} (P);

  \draw[to] (P) to node[auto] {b} (v1);
  \draw[to, loop above]  (v1) to node[auto] {a} (v1);

  \draw[to] (v1) to node[auto] {j} (v2);
  \draw[to, loop left]  (v2) to node[auto] {k} (v2);

\end{tikzpicture}
\end{center}
\caption{Left Krieger cover of the shift considered in Example
  \ref{ex_gfc_justifying}.
  Note that the labelled graph is no longer a presentation of the same
  shift if the decomposable predecessor set $P = P_1 \cup P_2$ is
  removed.
}  
\label{fig_gfc_justifying}
\end{figure}

\noindent Together with Lemma \ref{lem_decomposable}, this example
motivates the following definition. 

\begin{definition}
The \emph{generalized left Fischer cover} $(G, \LL_G)$ of a sofic
shift $X$ is defined to be the subgraph of the left Krieger
cover induced by $G^0 = \{ P \in K^0 \mid  P \geq P', P \textrm{ non-decomposable} \}$. 
\end{definition}

\noindent The following proposition justifies the term generalized left Fischer cover.

\begin{prop}
\label{prop_justifying}
\qquad
\begin{enumerate}[(i)]
\item The generalized left Fischer cover of a sofic shift $X$
  is a left-resolving and predecessor-separated presentation of $X$. 
\item If $X$ is an irreducible sofic shift then the generalized left
  Fischer cover is equal to the left Fischer cover. 
\item If $X_1,X_2$ are sofic shifts with
  disjoint alphabets then the generalized left Fischer cover
  of $X_1 \cup X_2$ is the disjoint union of the generalized left
  Fischer covers of $X_1$ and $X_2$. 
\end{enumerate}
\end{prop}

\begin{proof}
Given $y^- \in X^-$, choose $x^+ \in X^+$ such that $y^- \in
P_\infty(x^+) = P$. By definition of the generalized left Fischer cover,
there exist vertices $P_1, \ldots, P_n \in G^0$ such that $P = \bigcup_{i=1}^n
P_i$. Choose $i$ such that $y^- \in P_i$.
By construction, the left Krieger cover contains a path labelled $y^-$ terminating at
$P_i$. Since $P_i \in G^0$, this is also a path in the generalized
left Fischer cover. This proves that the generalized left Fischer
cover is a presentation of $X^-$, and hence also a presentation of
$X$. Since the left Krieger cover is left-resolving and
predecessor-separated, so is the generalized left Fischer cover. 

Let $X$ be an irreducible sofic shift, and identify the left Fischer cover $(F, \LL_F)$ with the top irreducible component of the left Krieger cover $(K,\LL_K)$ \cite[Lemma 2.7]{krieger_sofic_I}.
By the construction of the generalized left Fischer cover, it follows that the left Fischer cover is a subgraph of the generalized left Fischer cover.
Let $x^+ \in X^+$ such that $P = P_\infty(x^+)$ is non-decomposable. Let $S \subseteq F^0$ be the set of vertices where a presentation of $x^+$ in $(F,\LL_F)$ can start. Then $P = \bigcup_{v \in S} P_\infty(v)$, so $P \in S \subseteq F^0$ by assumption. 

Since $X_1$ and $X_2$ have no letters in common, the left Krieger
cover of $X_1 \cup X_2$ is just the disjoint union of the left
Krieger covers of $X_1$ and $X_2$. The generalized left Fischer
cover inherits this property from the left Krieger cover.
\end{proof}

The shift consisting of two non-interacting copies of the even shift is a simple example where the generalized left Fischer cover is a proper subgraph of the left Krieger cover.

\begin{lem}
\label{lem_GLFC_essential}
Let $X$ be a sofic shift with left Krieger cover $(K, \LL_K)$. If there
is an edge labelled $a$ from a non-decomposable $P \in K^0$ to a
decomposable $Q \in K^0$ then there exists a non-decomposable $Q' \in
K^0$ and an edge labelled $a$ from $P$ to $Q'$.   
\end{lem}

\begin{proof}
Choose $x^+ \in X^+$ such that $P = P_\infty(a x^+)$ and $Q =
P_\infty(x^+)$. Since $Q$ is decomposable, there exist $n > 1$ and
non-decomposable $Q_1, \ldots, Q_n \in K^0 \setminus \{ Q \}$
such that $Q = Q_1 \cup \cdots \cup Q_n$.
Let $S$ be the set of predecessor sets $P' \in K^0$ for which
there is an edge labelled $a$ from $P'$ to $Q_j$ for some $1 \leq j
\leq n$.
Given $y^- \in P$, $y^-ax^+ \in X$, so $y^-a \in Q$.
Choose $1 \leq i \leq n$ such that $y^-a \in Q_i$.
By construction, there exists $P' \in S$ such that $y^- \in
P'$. Reversely, if $y^- \in P' \in S$ then there is an edge labelled
$a$ from $P'$ to $Q_i$ for some $1 \leq i \leq n$, so $y^- a \in Q_i
\subseteq Q$. This implies that $y^-ax^+ \in X$, so $y^- \in P$.
Thus $P = \bigcup_{P' \in S} P'$, but $P$ is non-decomposable, so this
means that $P \in S$.
Hence, there is an edge labelled $a$ from $P$ to $Q_i$ for som $i$,
and $Q_i$ is non-decomposable. 
\end{proof}

\noindent The following proposition is an immediate consequence of this
result and the definition of the generalized left Fischer cover.

\begin{prop}
\label{prop_GLFC_essential}
The generalized left Fischer cover is essential.
\end{prop}

The left Fischer cover of an irreducible sofic shift $X$ is minimal
in the sense that no other left-resolving presentation of $X$
has fewer vertices. This not always the case for the generalized left Fischer cover.

\subsubsubsection{Canonical}
Krieger proved that a conjugacy $\Phi \colon X_1
\to X_2$ between sofic shifts with left Krieger covers $(K_1, \LL_1)$ and $ (K_2,\LL_2)$, respectively, induces a conjugacy
$\varphi
\colon \X_{K_1} \to \X_{K_2}$ such that $\Phi \circ
\pi_1 = \pi_2 \circ \varphi$ when $\pi_i \colon \X_{K_i} \to X_i$
is the covering map of the left Krieger cover of $X_i$ \cite{krieger_sofic_I}. A cover with
this property is said to be \emph{canonical}.
The next goal is to prove that the generalized left Fischer cover is canonical.
This will be done by using results and methods used by Nasu
\cite{nasu} to prove that the left Krieger cover is canonical. 

\begin{definition}[Bipartite code] \label{def_bipartite_code}
When $\AA, \CC, \DD$ are alphabets, an injective map $f \colon \AA \to
\CC \DD$ is called a \emph{bipartite expression}. If $X_1, X_2$ are
shift spaces with alphabets $\AA_1$ and $\AA_2$, respectively, and if $f_1
\colon \AA_1 \to \CC \DD$ is a bipartite expression then a map $\Phi
\colon X_1 \to X_2$ is said to be a \emph{bipartite code induced by}
$f_1$ if there exists a bipartite expression $f_2 \colon \AA_2 \to \DD
\CC$ such that one of the following two conditions is satisfied:
\begin{enumerate}[(i)]
\item If $x \in X_1$, $y = \Phi(x)$, and $f_1(x_i) = c_i d_i$ with $c_i
  \in \CC$ and $d_i \in \DD$ for all $i \in \Z$ then $f_2(y_i) = d_i
  c_{i+1}$ for all $i \in \Z$. 
\item If $x \in X_1$, $y = \Phi(x)$, and $f_1(x_i) = c_i d_i$ with $c_i
  \in \CC$ and $d_i \in \DD$ for all $i \in \Z$ then $f_2(y_i) = d_{i-1}
  c_i$ for all $i \in \Z$.
\end{enumerate}
A mapping $\Phi \colon X_1 \to X_2$ is called a \emph{bipartite
  code}, if it is the bipartite code induced by some bipartite
expression.
\end{definition}

\noindent It is clear that a bipartite code is a conjugacy and that the inverse
of a bipartite code is a bipartite code.

\begin{thm}[Nasu {\cite[Thm. 2.4]{nasu}}]
\label{thm_nasu_decomposition}
Any conjugacy between shift spaces can be decomposed into a
product of bipartite codes. 
\end{thm}

Let $\Phi \colon X_1 \to X_2$ be a bipartite code corresponding to
bipartite expressions $f_1 \colon \AA_1 \to \CC \DD$ and $f_2
\colon \AA_2 \to \DD \CC$, and use the bipartite expressions to recode
$X_1$ and $X_2$ to
\begin{align*}
  \hat X_1 &= \{ (f_1(x_i))_i \mid x \in X_1 \} \subseteq (\CC \DD)^\Z \\
  \hat X_2 &= \{ (f_2(x_i))_i \mid x \in X_2 \} \subseteq (\DD \CC)^\Z.
\end{align*}
For $i \in \{1,2\}$, $f_i$ induces a one-block conjugacy from $X_i$
to $\hat X_i$, and $\Phi$ induces a bipartite code $\hat \Phi \colon \hat X_1
\to \hat X_2$ which commutes with these conjugacies.
If $\Phi$ satisfies condition (i) in the definition of a bipartite
code then $(\hat \Phi (\hat x))_i = d_i c_{i+1}$ when $\hat x = (c_i d_i)_{i \in \Z}
\in \hat X_1$.
If it satisfies condition (ii) then $(\hat \Phi (\hat x))_i = d_{i-1}
c_i$ when $\hat x = (c_i d_i)_{i\in \Z} \in \hat X_1$. 
The shifts $\hat X_1$ and $\hat X_2$ will be called the \emph{recoded shifts}
of the bipartite code, and $\hat \Phi$ will be called the \emph{recoded
  bipartite code}.


A labelled graph $(G, \LL)$ is said to be \emph{bipartite} if $G$ is a
bipartite graph (i.e. the vertex set can be partitioned into two sets
$(G^0)_1$ and $(G^0)_2$ such that no edge has its range and source in
the same set).
When $(G, \LL)$ is a bipartite labelled 
graph over an alphabet $\AA$, define two graphs $G_1$ and $G_2$ as
follows: For $i \in \{1,2\}$, the vertex set of $G_i$ is $(G^0)_i$,
the edge set is 
the set of paths of length 2 in $(G, \LL)$ for which both range
and source are in $(G^0)_i$, and the range and source maps are
inherited from $G$.
For $i \in \{1,2\}$, define $\LL_i \colon G_i^1 \to \AA^2$ by
$\LL_i(ef) = \LL(e) \LL(f)$.
The pair $(G_1, \LL_1)$, $(G_2, \LL_2)$ is called the \emph{induced
pair of labelled graphs} of $(G, \LL)$. This decomposition is not
necessarily unique, but whenever a bipartite labelled graph is
mentioned, it will be assumed that the induced graphs are specified. 

\begin{remark}[Nasu {\cite[Remark 4.2]{nasu}}]
\label{rem_standard_code_of_bipartite_graph}
Let $(G, \LL)$ be a bipartite labelled graph for which the induced
pair of labelled graphs is $(G_1, \LL_1)$, $(G_2, \LL_2)$. Let
$X_1$ and $X_2$ be the sofic shifts presented by these graphs, and
let $\X_{G_1}, 
\X_{G_2}$ be the edge shifts generated by $G_1$, $G_2$.
The natural embedding $f \colon G_1^1 \to (G^1)^2$ is a bipartite
expression which induces two bipartite codes $\varphi_\pm \colon \X_{G_1}
\to \X_{G_2}$ such that $(\varphi_+(x))_i = f_i e_{i+1}$ and
$(\varphi_-(x))_i = f_{i-1} e_i$ when $x = (e_i f_i)_{i \in \Z} \in
\X_{G_1}$.     
Similarly, the embedding $F \colon \LL_1(G_1^1) \to (\LL(G^1))^2$ is a
bipartite expression which induces bipartite codes $\Phi_\pm \colon
X_1 \to X_2$ such that
$(\Phi_+(x))_i = b_i a_{i+1}$ and $(\Phi_-(x))_i = b_{i-1} a_i$ when
$x = (a_i b_i)_{i \in \Z} \in X_1$.
By definition, 
$\Phi_\pm \circ \pi_1 = \pi_2 \circ \varphi_\pm$ when 
$\pi_1 \colon \X_{G_1} \to X_1$, $\pi_2 \colon \X_{G_2} \to X_2$ are
the covering maps.
The bipartite codes $\varphi_\pm$ and $\Phi_\pm$ are called the
\emph{standard bipartite codes induced by} $(G, \LL)$. 
\end{remark}

 
\begin{lem}[Nasu {\cite[Cor. 4.6 (1)]{nasu}}]
\label{lem_nasu_4.6}
Let $\Phi \colon X_1 \to X_2$ be a bipartite code between sofic
shifts $X_1$ and $X_2$.
Let $\hat X_1$ and $\hat X_2$ be the recoded shifts of $X_1$ and $X_2$
respectively, and let $(K_1, \LL_1)$ and $(K_2, \LL_2)$ be the left
Krieger covers of $\hat X_1$ and $\hat X_2$ respectively.
Then there exists a sofic shift $\hat X$ for which the left Krieger cover
is a bipartite labelled graph such that the induced pair of labelled
graphs is $(K_1, \LL_1)$, $(K_2, \LL_2)$ and such that the recoded
bipartite code $\hat \Phi \colon \hat X_1 \to \hat X_2$ of $\Phi$ is
one of the standard bipartite codes $\Phi_\pm$ induced by the left Krieger
cover of $\hat X$ as defined in Remark
\ref{rem_standard_code_of_bipartite_graph}. 
\end{lem}

The proof of the following theorem is very similar to the proof of the
corresponding result by Nasu \cite[Thm. 3.3]{nasu} for the left Krieger cover.

\begin{thm} \label{thm_GFC_canonical}
The generalized left Fischer cover is canonical.
\end{thm}

\begin{proof}
Let $\Phi \colon X_1 \to X_2$ be a bipartite code. 
Let $\hat X_1, \hat X_2$ be the recoded shifts, let $(K_1, \LL_1)$, $(K_2,
\LL_2)$ be the corresponding left Krieger covers, and let $\hat \Phi
\colon \hat X_1 \to \hat X_2$ be the recoded bipartite code. 
Use Lemma \ref{lem_nasu_4.6} to find a sofic shift $\hat X$ such that the
left Krieger cover $(K, \LL)$ of $\hat X$ is a bipartite labelled
graph for which 
the induced pair of labelled graphs is $(K_1, \LL_1)$, $(K_2, \LL_2)$.
Let $(G_1, \LL_1)$, $(G_2, \LL_2)$, and $(G, \LL)$ be the
generalized left Fischer covers of respectively $\hat X_1$, $\hat X_2$, and
$\hat X$.

The labelled graph $(G, \LL)$ is bipartite since $G$ is a
subgraph of $K$. Note that a predecessor set $P$ in $K_1^0$ or $K_2^0$ is
decomposable if and only if the corresponding predecessor set in $K^0$
is decomposable. 
If $i \in \{1,2\}$ and $Q \in G_i^0 \subseteq K_i^0$ then there is a
path in $K_i$ from $Q$ to a non-decomposable $P \in K_i^0$.
By considering the corresponding path in $K$, it is clear that the
vertex in $K^0$ corresponding to $Q$ is in $G^0$.
Conversely, if $Q \in G^0$ then there is a path in $K$ from $Q$ to a
non-decomposable $P \in K^0$.
If $P$ and $Q$ belong to the same partition $K_i^0$ then the vertex in
$K_i$ corresponding to $Q$ is in $G_i^0$ by definition.
On the other hand, if $Q$ corresponds to a vertex in $K_i$ and if $P$ 
belongs to the other partition then Lemma \ref{lem_GLFC_essential}
shows that there exists a non-decomposable $P'$ in the same partition
as $Q$ and an edge from $P$ to $P'$ in $K$. Hence, there is also a
path in $K_i$ from the vertex corresponding to $Q$ to the vertex
corresponding to $P'$, so $Q \in G_i^0$. This proves that the pair of
induced labelled graphs of $(G, \LL)$ is $(G_1, \LL_1)$, $(G_2,
\LL_2)$.   

Let $\hat \Psi_\pm \colon \hat X_1 \to \hat X_2$ be the standard
bipartite codes induced by $(G, \LL)$.
Remark \ref{rem_standard_code_of_bipartite_graph} shows that
there exist bipartite codes $\hat \psi_\pm \colon \X_{G_1} \to \X_{G_2}$
such that $\hat \Psi_\pm \circ \hat \pi_1|_{\X_{G_1}} = \hat
\pi_2|_{\X_{G_2}} \circ \hat \psi_\pm$.
The labelled graph $(G, \LL)$ presents the same sofic shift as $(K,
\LL)$, so they both induce the same standard bipartite codes from
$\hat X_1$ to $\hat X_2$, and by Lemma \ref{lem_nasu_4.6},
$\hat \Phi$ is one of these standard bipartite codes, so $\hat
\Phi = \hat \Psi_+$ or $\hat \Phi = \hat \Psi_-$. In particular, there
exists a bipartite code $\hat \psi \colon \X_{G_1} \to \X_{G_2}$ such
that $\hat \Phi \circ \hat \pi_1|_{\X_{G_1}} = \hat \pi_2|_{\X_{G_2}}
\circ \hat \psi$. 

By recoding $\hat X_1$ to $X_1$ and $\hat X_2$ to $X_2$ via
the bipartite expressions inducing $\Phi$,
this gives a bipartite code $\psi$ 
such that $\Phi \circ \pi_1 = \pi_2 \circ \psi$ when $\pi_1, \pi_2$
are the covering maps of the generalized left Fischer covers of $X_1$
and $X_2$ respectively.
By Theorem \ref{thm_nasu_decomposition}, any conjugacy can be
decomposed as a product of bipartite codes, so this proves that the
generalized left Fischer cover is canonical. 
\end{proof}


\begin{thm} \label{thm_GFC_FE}
If $X_1, X_2$ are flow equivalent sofic shifts with genera\-lized left
Fischer covers $(G_1, \LL_1)$ and $(G_2, \LL_2)$, respectively,   
then the covering maps $\pi_1 \colon \X_{G_1} \to X_1$ and $\pi_2
\colon \X_{G_2} \to X_2$ are flow equivalent, i.e. there exist flow
equivalences $\varphi \colon X_1 \to X_2$ and $\psi \colon \X_{G_1}
\to \X_{G_2}$ such that $\varphi \circ \pi_1= \pi_2 \circ \psi$
\end{thm}

\begin{proof}
In \cite{boyle_carlsen_eilers} it is proved that the left Krieger
cover respects symbol expansion: If $X$ is a sofic shift with
alphabet $\AA$, $a \in \AA$, $\bullet$ is some symbol not in $\AA$, and
if $\hat X$ is obtained from $X$ via a symbol expansion which
inserts a $\bullet$ after each $a$ then the left Krieger cover of $\hat
X$ is obtained by replacing each edge labelled $a$ in the left
Krieger cover of $X$ by two edges in sequence labelled $a$ and $\bullet$
respectively. Clearly, the generalized left Fischer cover inherits
this property. By \cite{boyle_carlsen_eilers}, any canonical cover
which respects flow equivalence has the desired property, so the result follows from Theorem \ref{thm_GFC_canonical}.
\end{proof}

\section{Foundations and layers of covers}
\label{sec_foundation}

Let $\EE = (E, \LL)$ be a finite left-resolving and predecessor-separated labelled graph.
For each $V \subseteq E^0$ and each word $w$ over the alphabet $\AA$ of $\LL$  define 
\begin{displaymath}
   wV = \{ u \in E^0 \mid u \textrm{ is the source of a path labelled } w \textrm{ terminating in } V \}.
\end{displaymath}

\begin{definition}
Let $S$ be a subset of the power set $\PP(E^0)$, and let ${\sim}$ be an equivalence relation on $S$.
The pair $(S,{\sim})$ is said to be \emph{past closed} if 
\begin{itemize}
\item $\{ v \} \in S$,
\item $\{ u \} \sim \{ v \}$ implies $u = v$,
\item $aV \neq \emptyset$ implies $aV \in S$, and
\item $U \sim V$ and $aU \neq  \emptyset$ implies $aV \neq \emptyset$ and $aU \sim aV$
\end{itemize}
for all $u,v \in E^0$, $U,V \in S$, and $a \in \AA$.
\end{definition}

Let $(S, {\sim})$ be past closed. For each $V \in S$, let $\class V$ denote the equivalence class of $V$ with respect to ${\sim}$.
When $a \in \AA$ and $V \in S$, $\class V$ is said to receive $a$ if $aV \neq \emptyset$.
For each $\class V \in S / {\sim}$, define $\lvert \class V \rvert = \min_{V \in \class V} \lvert V \rvert$. 

\begin{definition}
Define $\GG(\EE, S, {\sim})$ to be the labelled graph with vertex set $S / {\sim}$ for which there is an edge labelled $a$ from $\class{aV}$ to $\class V$ whenever $\class V$ receives $a$.
For each $n \in \N$, the $n$\emph{th layer} of $\GG(\EE, S,{\sim})$ is the labelled subgraph induced by $S_n = \{ \class V \in S / {\sim} \mid  n = \lvert \class V \rvert \}$. 
$\EE$ is said to be a \emph{foundation} of any labelled graph isomorphic to $\GG(\EE, S, {\sim})$.
 \end{definition}

If a labelled graph $\HH$ is isomorphic to $\GG(\EE, S, {\sim})$ then the subgraph of $\HH$ corresponding to the $n$th layer of $\GG(\EE, S, {\sim})$ is be said to be the $n$\emph{th layer of} $\HH$ \emph{with respect to} $\EE$, or simply the $n$th layer if $\EE$ is understood from the context.

\begin{prop}
$\EE$ and $\GG(\EE, S,{\sim})$ present the same sofic shift, and
$\EE$ is labelled graph isomorphic to the first layer of $\GG(\EE, S,{\sim})$. 
\end{prop}

\begin{proof}
By assumption, there is a bijection between $E^0$ and the set of vertices in the first layer of $\GG(\EE, S,{\sim})$. By construction, there is an edge labelled $a$ from $u$ to $v$ in $\EE$ if and only if there is an edge labelled $a$ from $\class{\{u\}}$ to $\class{\{v\}}$ in $\GG(\EE, S,{\sim})$.
Every finite word presented by $\GG(\EE, S, {\sim})$ is also presented by $\EE$, so they present the same sofic shift.
\end{proof}

The following proposition motivates the use of the term layer by showing that edges can never go from higher to lower layers.

\begin{prop}
\label{prop_layers}
If $\class V \in S/ {\sim}$ receives $a \in \AA$ then $\lvert \class{aV} \rvert  \leq{} \lvert \class V \rvert$. If
$\GG(\EE, S,{\sim})$ has an edge from a vertex in the $m$th layer to a vertex in the $n$th layer then $m \leq n$.
\end{prop}

\begin{proof}
Choose $V \in \class V$ such that $\lvert V \rvert = \lvert \class V \rvert$. Each $u \in aV$ emits at least one edge labelled $a$ terminating in $V$, and $\EE$ is left-resolving, so $\lvert \class{aV} \rvert \leq \lvert aV \rvert \leq \lvert V \rvert = \lvert \class V \rvert$. The second statement follows from the definition of $\GG(\EE, S,{\sim})$.
\end{proof}

\begin{example}
Let $(F, \LL_F)$ be the left Fischer cover of an irreducible sofic shift $X$. 
For each $x^+ \in X^+$, define $s(x^+) \subseteq F^0$ to be the set of 
vertices where a presentation of $x^+$ can start. $S = \{ s(x^+) \mid x^+ \in X^+ \} \subseteq \PP(F^0)$ is past closed since each vertex in the left Fischer cover is the predecessor set of an intrinsically synchronizing right-ray, so the \emph{multiplicity set cover} of $X$ can be defined to be $\GG((F, \LL_F), S, =)$.
 An analogous cover can be defined by considering the vertices where presentations of finite words can start.
Thomsen \cite{thomsen} constructs the
derived shift space $\partial X$ of $X$
using
right-resolving graphs, but an analogous construction works for left-resolving graphs.
The procedure from \cite[Example 6.10]{thomsen} shows that this
$\partial X$ is presented by the labelled graph obtained by removing the left Fischer cover from the multiplicity set cover.
\end{example}

Let $X$ be a sofic shift, and let $(K, \LL_K)$ be the left Krieger cover of $X$. In order to use the preceding results to investigate the structure of the left Krieger cover and the past set cover, define an equivalence relation on $\PP(K^0)$ by $U \sim_\cup V$ if and only if $\bigcup_{P \in U} P = \bigcup_{Q \in V} Q$. Clearly, $\{ P \} \sim_\cup \{ Q \}$ if and only if $P = Q$. If $U, V \subseteq K^0$, $a \in \AA$, $aV \neq \emptyset$, and $U \sim_\cup V$ then $aU \sim_\cup aV$ by the definition of the left Krieger cover.

\begin{thm}
\label{thm_gfc_foundation_lkc}
For a sofic shift $X$, the generalized left Fischer cover $(G, \LL_G)$ is a foundation of the left Krieger cover $(K, \LL_K)$, and no smaller subgraph is a foundation.
\end{thm}

\begin{proof}
Define $S = \{ V \subseteq G^0 \mid \exists x^+ \in X^+ \textrm{ such that } P_\infty(x^+) = \bigcup_{P \in V} P \}$. 
Note that $\{ P \} \in S$ for every $P \in G^0$.
If $x^+ \in X^+$ with $P_\infty(x^+) = \bigcup_{P \in V} P$ and if 
$aV \neq \emptyset$ for some $a \in \AA$ then $ax^+ \in X^+$ and $P_\infty(ax^+) = \bigcup_{P \in aV} P$.
This proves that  the pair $(S, {\sim_\cup})$ is past closed, so 
$\GG((G,\LL_G), S, {\sim_\cup})$ is well defined.
Since $(G, \LL_G)$ is a presentation of $X$, there is a bijection $\varphi \colon S / {\sim_\cup} \to K^0$ defined by $\varphi( \class V ) = \bigcup_{P \in V} P$.
By construction, there is an edge labelled $a$ from $\class U$ to $\class V$ in $\GG((G,\LL_G), S, {\sim_\cup})$ if and only if there exists $x^+ \in X^+$ such that
$P_\infty(ax^+) = \bigcup_{P \in U} P$ and $P_\infty(x^+) = \bigcup_{Q \in V} Q$, so $\GG((G,\LL_G), S, {\sim_\cup})$ is isomorphic to $(K, \LL_K)$.
It follows from Lemma \ref{lem_decomposable} that no proper subgraph of $(G, \LL_G)$ can be a foundation of the left Krieger cover.
\end{proof}

The example from \cite[Section 4]{carlsen_matsumoto} shows that the left Krieger cover can be a proper subgraph of the past set cover. The following lemma will be used to further investigate this relationship.

\begin{lem}
\label{lem_prefix_with_same_past}
Let $X$ be a sofic shift. For every right-ray $x^+ = x_1 x_2 x_3 \ldots \in X^+$ there exists $n \in \N$ such that $P_\infty(x^+) = P_\infty(x_1 x_2 \ldots x_k)$ for all $k \geq n$. 
\end{lem}

\begin{proof} 
It is clear that $P_\infty(x_1) \supseteq P_\infty(x_1 x_2) \supseteq
\cdots \supseteq P_\infty(x^+)$. Since $X$ is
sofic, there are only finitely many different predecessor sets of
words, so there must exist $n \in \N$ such that $P_\infty(x_1 x_2 \ldots x_k)
= P_\infty(x_1 x_2 \ldots x_n)$ for all $k \geq n$. If $y^- \in
P_\infty(x_1 x_2 \ldots x_n)$ is given, then $y^- x_1 x_2 \ldots x_k \in X$ for
all $k \geq n$, so $y^- x^+$ contains no forbidden words, and
therefore $y^- \in P_\infty(x^+)$. Since $y^-$ was arbitrary,
$P_\infty(x^+) =  P_\infty(x_1 x_2 \ldots x_n)$.
\end{proof}

\begin{thm}
\label{thm_foundations_psc}
For a sofic shift $X$, the generalized left Fischer cover $(G, \LL_G)$ and the left Krieger cover $(K, \LL_K)$ are both foundations of the past set cover $(\psc, \LL_\psc)$.
\end{thm}

\begin{proof}
Define $S = \{ V \subseteq G^0 \mid \exists w \in \BB(X) \textrm{ such that } P_\infty(w) = \bigcup_{P \in V} P \}$,
and use Lemma \ref{lem_prefix_with_same_past} to conclude that $S$ contains $\{ P \}$ for every $P \in G^0$. By arguments analogous to the ones used in the proof of Theorem \ref{thm_gfc_foundation_lkc}, it follows that $\GG((G,\LL_G),S, {\sim_\cup})$ is isomorphic to $(\psc, \LL_\psc)$. 
To see that $(K, \LL_K)$ is also a foundation, define 
$T = \{ V \subseteq K^0 \mid \exists w \in \BB(X) \textrm{ such that } P_\infty(w) = \bigcup_{P \in V} P \}$,
 and apply arguments analogous to the ones used above to prove that $(\psc, \LL_\psc)$ is isomorphic to $\GG((K,\LL_K),T, {\sim_\cup})$.
\end{proof}

In the following, the $n$th layer of the left Krieger cover (past set cover) will always refer to the $n$th layer with respect to the generalized left Fischer cover $(G, \LL_G)$. For a right-ray (word) $x$, $P_\infty(x)$ is a vertex in the $n$th layer of the left Kriger cover (predecessor set cover) for some $n \in \N$, and such an $x$ is  said to be $\frac{1}{n}$-\emph{synchronizing}. Note that $x$ is $\frac{1}{n}$-synchronizing if and only if $n$ is the smallest number such that there exist $P_1, \ldots, P_n \in G^0$ with $\bigcup_{i=1}^n P_i = P_\infty(x)$.
In an irreducible sofic shift with left Fischer cover $(F, \LL_F)$, this happens if and only if $n$ is the smallest number such that there exist $u_1, \ldots, u_n \in F^0$ with $\bigcup_{i=1}^n P_\infty(u_i) = P_\infty(x)$.

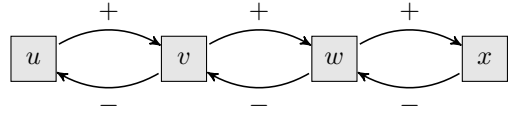
\begin{figure}
\begin{center}
\begin{tikzpicture}
  [
   knude/.style = {circle, inner sep = 0pt},
   vertex/.style = {circle, draw, minimum size = 1 mm, inner sep =
      0pt, fill=black!10},
   textVertex/.style = {rectangle, draw, minimum size = 6 mm, inner sep =
      1pt, fill=black!10},
   to/.style = {->, shorten <= 1 pt, >=stealth', semithick}]
  

    \node[textVertex] (u) at (-3,4) {$u$};  
    \node[textVertex] (v) at (-1,4) {$v$};  
    \node[textVertex] (w) at (1,4) {$w$};  
    \node[textVertex] (x) at (3,4) {$x$}; 

  \draw[to, bend left=30] (u) to node[auto] {$+$} (v);
  \draw[to, bend left=30] (v) to node[auto] {$+$} (w);
  \draw[to, bend left=30] (w) to node[auto] {$+$} (x);
  \draw[to, bend left=30] (x) to node[auto] {$-$} (w);
  \draw[to, bend left=30] (w) to node[auto] {$-$} (v);
  \draw[to, bend left=30] (v) to node[auto] {$-$} (u);

\end{tikzpicture}
\end{center}
\caption{Left Fischer cover of the 3-charge
  constrained shift.} 
\label{F_LFC_3cc}
\end{figure}

\begin{figure}
\begin{center}
\begin{tikzpicture}
  [
   knude/.style = {circle, inner sep = 0pt},
   vertex/.style = {circle, draw, minimum size = 1 mm, inner sep =
      0pt, fill=black!10},
   textVertex/.style = {rectangle, draw, minimum size = 6 mm, inner sep =
      1pt, fill=black!10},
   to/.style = {->, shorten <= 1 pt, >=stealth', semithick}]
  
    \node[textVertex] (Pu) at (-4.5,1) {$P_\infty(u)$};
    \node[textVertex] (Pv) at (-1.5,1) {$P_\infty(v)$};
    \node[textVertex] (Pw) at ( 1.5,1) {$P_\infty(w)$}; 
    \node[textVertex] (Px) at ( 4.5,1) {$P_\infty(x)$}; 

    \node[textVertex] (Puv) at (-4,-1) {$P_\infty(u) \cup P_\infty(v)$};
    \node[textVertex] (Pvw) at ( 0,-1) {$P_\infty(v) \cup P_\infty(w)$};
    \node[textVertex] (Pwx) at ( 4,-1) {$P_\infty(w) \cup P_\infty(x)$};

    \node[textVertex] (Puvw) at (-3,-3) {$P_\infty(u) \cup
      P_\infty(v) \cup P_\infty(w)$};
    \node[textVertex] (Pvwx) at (3,-3) {$P_\infty(v) \cup
      P_\infty(w) \cup P_\infty(x)$};

  \draw[to, bend left=20] (Pu) to node[auto] {$+$} (Pv);
  \draw[to, bend left=20] (Pv) to node[auto] {$+$} (Pw);
  \draw[to, bend left=20] (Pw) to node[auto] {$+$} (Px);
  \draw[to, bend left=20] (Px) to node[auto] {$-$} (Pw);
  \draw[to, bend left=20] (Pw) to node[auto] {$-$} (Pv);
  \draw[to, bend left=20] (Pv) to node[auto] {$-$} (Pu);
 
  \draw[to] (Pu) to node[auto,swap] {$+$} (Puv);
  \draw[to] (Px) to node[auto] {$-$} (Pwx);

  \draw[to, bend left=10] (Puv) to node[auto] {$+$} (Pvw);
  \draw[to, bend left=10] (Pvw) to node[auto] {$+$} (Pwx);
  \draw[to, bend left=10] (Pwx) to node[auto] {$-$} (Pvw);
  \draw[to, bend left=10] (Pvw) to node[auto] {$-$} (Puv);

  \draw[to] (Puv) to node[auto,swap] {$+$} (Puvw);
  \draw[to] (Pwx) to node[auto] {$-$} (Pvwx);

  \draw[to, bend left=5] (Puvw) to node[auto] {$+$} (Pvwx);
  \draw[to, bend left=5] (Pvwx) to node[auto] {$-$} (Puvw);
\end{tikzpicture}
\end{center}
\caption{Left Krieger cover of the 3-charge
  constrained shift.} 
\label{F_LKC_3cc}
\end{figure}
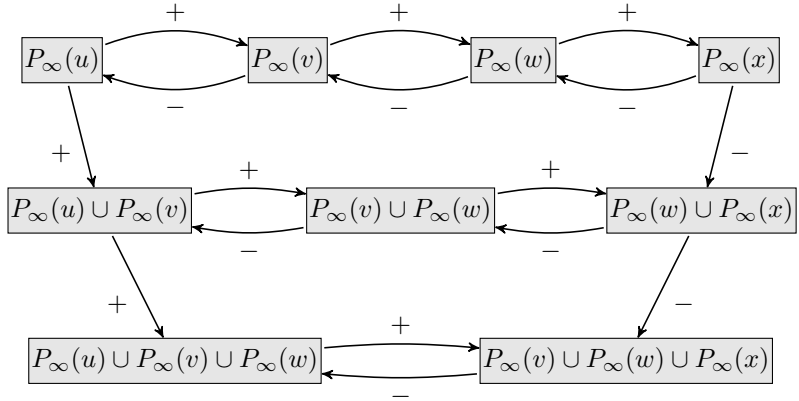

\begin{example}
Figures \ref{F_LFC_3cc} and \ref{F_LKC_3cc} show, respectively, the left Fischer and the left Krieger cover of the 3-charge constrained 
shift (see e.g. \cite[1.2.7]{lind_marcus} for the definition of charge
constrained shifts).
There are 3 vertices in the second layer of the left Krieger cover and
two in the third.
Note how the left Fischer cover can be identified with the first layer
of the left Krieger cover.
Note also that the second layer is the left Fischer cover
of the 2-charge constrained shift and that the third layer is the left
Fischer cover of the 1-charge constrained shift.
\end{example}

\begin{cor}
\label{cor_PSC_reducible_if_LKC_reducible}
If the left Krieger cover of a sofic shift is reducible then so is the past set cover.
\end{cor}
\nopagebreak
\begin{proof}
This follows from Proposition \ref{prop_layers} and Theorem  \ref{thm_foundations_psc}. 
\end{proof}

\section{The range of a flow invariant}
\label{sec_invariant}
Let $E$ be a directed graph.
Vertices $u,v \in E^0$ \emph{properly communicate} \cite{bates_eilers_pask}
if there are paths
$\mu, \lambda \in E^*$ of length greater than or equal to 1 such that
$s(\mu) = u$, $r(\mu) = v$, $s(\lambda) = v$, and $r(\lambda) =
u$. 
This relation is used to construct maximal disjoint subsets of $E^0$, called \emph{proper communication sets of vertices}, such
that $u,v \in E^0$ properly communicate if and only if they belong to
the same subset. 
The \emph{proper communication graph} $PC(E)$ is defined to be
the directed graph for which the vertices are the proper
communication sets of vertices of $E$ and for which there is an edge
from one proper communication set to another if and only if there is
a path from a vertex in the first set to a vertex in the second. 
The proper communication graph of the left Krieger cover of a sofic shift space is a flow-invariant \cite{bates_eilers_pask}.

Let $X$ be an irreducible sofic shift with left Fischer cover $(F, \LL_F)$ and left Krieger cover $(K, \LL_K)$, and let $E$ be the proper communication graph of $K$. By construction, $E$ is finite and contains no circuit. The left Fischer cover  is isomorphic to an irreducible subgraph of $(K,\LL_K)$ corresponding to a root $r \in E^0$ \cite[Lemma 2.7]{krieger_sofic_I}, and by definition, there is an edge from $u \in E^0$ to $v \in E^0$ whenever $u > v$. 
The following proposition gives the range of the flow-invariant 
by proving that all such graphs can occur. 

\begin{prop}
\label{prop_range_invariant}
Let $E$ be a finite directed graph with a root and without circuits. $E$ is the proper communication graph of the left
Krieger cover of an AFT shift if there is an edge from $u \in E^0$ to $v \in E^0$ whenever $u > v$.
\end{prop}

\begin{proof} 
Let $E$ be an arbitrary finite directed graph which contains no circuit and which has a root $r$, and let $\tilde E$ be the directed graph obtained from $E$ by adding an edge from $u \in E^0$ to $v \in E^0$ whenever $u > v$. 
The goal is to construct a labelled graph $(F, \LL_F)$ which is the
left Fischer cover of an irreducible sofic shift with the desired
properties. 
For each $v \in E^0$, let $l(v)$ be the length of the
longest path from $r$ to $v$. This is well-defined since $E$
does not contain any circuits.
For each $v \in E^0$, define $n(v) = 2^{l(v)}$ vertices $v_1, \ldots ,
v_{n(v)}  \in F^0$. The single vertex corresponding to the root $r
\in E^0$ is denoted $r_1$.
For each $v \in E^0$, draw a loop of length 1 labelled $a_v$ at each
of the vertices $v_1, \ldots , v_{n(v)}  \in F^0$. 
If there is an edge from $u \in E^0$ to $v \in E^0$ then $l(v) >
l(u)$. From each vertex $u_1, \ldots , u_{n(u)}$ draw
$n(u,v) = \frac{n(v)}{n(u)} = 2^{l(v)-l(u)} \geq 2$ edges labelled $a_{u,v}^1,
\ldots , a_{u,v}^{n(u,v)}$ such that every vertex $v_1,
\ldots , v_{n(v)}$ receives exactly one of these edges. 
For each sink $v \in E^0$ draw a uniquely labelled edge from each
vertex $v_1, \ldots , v_{n(v)}$ to $r_1$. 
This finishes the construction of $(F, \LL_F)$.

By construction, $F$ is irreducible, right-resolving, and
left-resolving. Additionally, it is 
predecessor-separated because there is a uniquely labelled path to every
vertex in $F^0$ from $r_1$.
Thus,  $(F, \LL_F)$ is the left Fischer
cover of an AFT shift $X$. Let $(K, \LL_K)$ be the
left Krieger cover of $X$. 

For every $v \in E^0$, $P_\infty(a_v^\infty) = \bigcup_{i=1}^{n(v)}
P_\infty(v_i)$ and no smaller set of vertices has this property, so
$P_\infty(a_v^\infty)$ is a vertex in the $n(v)$th layer of the left
Krieger cover. There is clearly a loop labelled $a_v$ at the
vertex $P_\infty(a_v^\infty)$, so it belongs to a proper
communication set of vertices. 
Furthermore, $b a_v^\infty \in X^+$ if and only if $b =
a_v$ or $b = a_{u,v}^i$ for some $u \in E^0$ and $1 \leq i \leq
n(u,v)$. By construction,   
$P_\infty(a_{u,v}^i a_v^\infty) = \bigcup_{i=1}^{n(u)} P_\infty(u_i) =
P_\infty(a_u^\infty)$, so there is an edge from
$P_\infty(a_u^\infty)$ to $P_\infty(a_v^\infty)$ if and only if there
is an edge from $u$ to $v$ in $E$. This proves that $E$, and hence also $\tilde E$, are a subgraphs
of the proper communication graph of $K$.

\begin{figure}
\begin{center}
\begin{tikzpicture}
  [bend angle=45,
   knude/.style = {circle, inner sep = 0pt},
   vertex/.style = {circle, draw, minimum size = 1 mm, inner sep =
      0pt, fill=black!10},
   textVertex/.style = {rectangle, draw, minimum size = 6 mm, inner sep =
      1pt, fill=black!10},
   to/.style = {->, shorten <= 1 pt, >=stealth', semithick}]
  
  \node[textVertex] (r) at (0,0) {$r$}; 
  \node[textVertex] (x) at (-1.5,-1.5) {$x$}; 
  \node[textVertex] (y) at (0,-1.5) {$y$}; 
  \node[textVertex] (z) at (1.5,-1.5) {$z$}; 

  \draw[to] (r) to (x);
  \draw[to] (r) to (y);
  \draw[to] (r) to (z);
  \draw[to] (y) to (z);

\end{tikzpicture}
\end{center}
\caption{A directed graph with root $r$ and without circuits.}
\label{fig_ex_construction_E}
\end{figure}
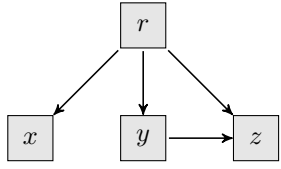  
\begin{figure}
\begin{center}
\begin{tikzpicture}
  [bend angle=5,
   knude/.style = {circle, inner sep = 0pt},
   vertex/.style = {circle, draw, minimum size = 1 mm, inner sep =
      0pt, fill=black!10},
   textVertex/.style = {rectangle, draw, minimum size = 6 mm, inner sep =
      1pt, fill=black!10},
   to/.style = {->, shorten <= 1 pt, >=stealth', semithick}]
  
  \node[textVertex] (v0) at (0,0) {$r_1$}; 

  \node[textVertex] (vx1) at (-7, -2) {$x_1$}; 
  \node[textVertex] (vx2) at (-6, -2) {$x_2$}; 
  \node[textVertex] (vy1) at (-2.6, -2) {$y_1$}; 
  \node[textVertex] (vy2) at (2.6, -2) {$y_2$}; 

  \node[textVertex] (vz1) at (-3.9, -4) {$z_1$}; 
  \node[textVertex] (vz2) at (-1.3, -4) {$z_2$}; 
  \node[textVertex] (vz3) at (1.3, -4) {$z_3$}; 
  \node[textVertex] (vz4) at (3.9, -4) {$z_4$}; 

  \node[textVertex] (v0') at (0,-7) {$r_1$}; 

  \draw[to, bend right = 20] (v0) to node[near end, above] {$a_{r,x}^1$} (vx1);
  \draw[to, bend right = 18] (v0) to node[near end, below] {$a_{r,x}^2$} (vx2);
  \draw[to, bend right] (v0) to node[near end, right] {$a_{r,y}^1$} (vy1);
  \draw[to, bend left] (v0) to node[near end, left] {$a_{r,y}^2$} (vy2);

  \draw[to, bend right = 40] (v0) to node[near end, left] {$a_{r,z}^1$} (vz1);
  \draw[to, bend right] (v0) to node[near end, right] {$a_{r,z}^2$} (vz2);
  \draw[to, bend left]  (v0) to node[near end, left] {$a_{r,z}^3$} (vz3);
  \draw[to, bend left= 40] (v0) to node[near end, right]  {$a_{r,z}^4$} (vz4);

  \draw[to, bend right] (vy1) to node[very near end, right] {$a_{y,z}^1$} (vz1);
  \draw[to, bend left]  (vy1) to node[very near end, left] {$a_{y,z}^2$} (vz2);
  \draw[to, bend right] (vy2) to node[very near end, right] {$a_{y,z}^1$} (vz3);
  \draw[to, bend left]  (vy2) to node[very near end, left]  {$a_{y,z}^2$} (vz4);

  \draw[to, loop below] (vx1) to node[auto] {$a_x$} (vx1);
  \draw[to, loop below] (vx2) to node[auto] {$a_x$} (vx2) ;
  \draw[to, loop right] (vy1) to node[auto] {$a_y$} (vy1);
  \draw[to, loop left] (vy2) to node[auto] {$a_y$} (vy2);
  \draw[to, loop below] (vz1) to node[auto] {$a_z$} (vz1);
  \draw[to, loop below] (vz2) to node[auto] {$a_z$} (vz2) ;
  \draw[to, loop below] (vz3) to node[auto] {$a_z$} (vz3);
  \draw[to, loop below] (vz4) to node[auto] {$a_z$} (vz4);

  \draw[to, bend right = 35] (vx1) to (v0');
  \draw[to, bend right = 30] (vx2) to (v0');
  \draw[to, bend right = 30] (vy1) to (v0');
  \draw[to, bend left = 30] (vy2) to (v0');
  \draw[to, bend right = 10] (vz1) to (v0');
  \draw[to, bend left = 10] (vz2) to (v0');
  \draw[to, bend right = 10] (vz3) to (v0');
  \draw[to, bend left = 10] (vz4) to (v0');
\end{tikzpicture}
\end{center}
\caption{Left Fischer cover of the sofic shift $X$ con\-sidered in
  Example \ref{ex_construction}.
 }
\label{fig_ex_construction_F}
\end{figure}
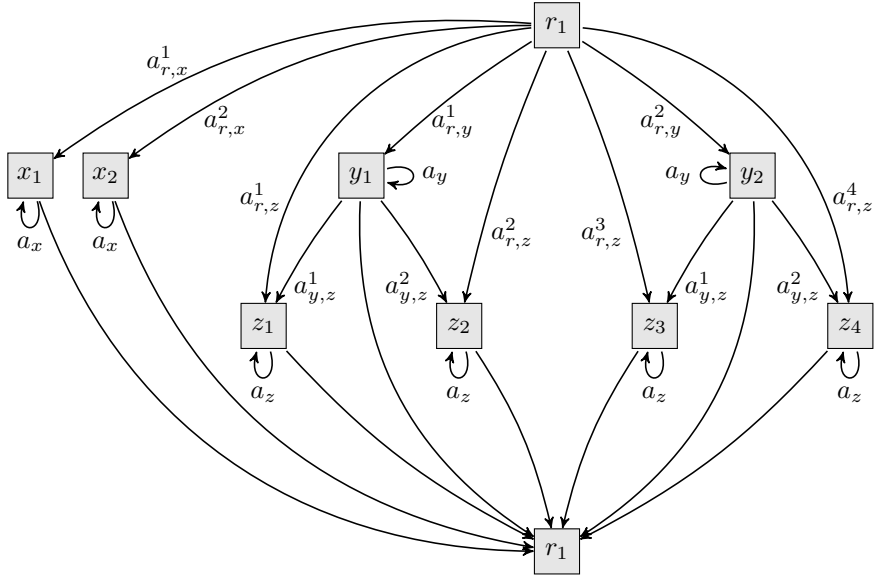 

Since the edges which terminate at $r_1$ are uniquely labelled, any
$x^+ \in X^+$ which contains one of these letters must be intrinsically
synchronizing. If $x^+ \in X^+$ does not contain any
of these letters then $x^+$ must be eventually periodic with $x^+ = w
a_v^\infty$ for some $v \in E^0$ and $w \in \BB(X)$. Thus, $K$ only
has the vertices described above, and therefore the proper
communication graph of $K$ is $\tilde E$. 
\end{proof}

\begin{example}
\label{ex_construction}
To illustrate the construction used in
the proof of Proposition \ref{prop_range_invariant}, let
$E$ be the directed graph drawn in Figure
\ref{fig_ex_construction_E}. $E$ has a unique
maximal vertex $r$ and contains no circuit, so 
it is the proper communication graph of the left Krieger cover of an
irreducible sofic shift.
Note that $l(x) = l(y) = 1$ and that $l(z) = 2$. Figure
\ref{fig_ex_construction_F} shows the left Fischer cover of a sofic
shift $X$ constructed using the method from the proof of Proposition
\ref{prop_range_invariant}. Note that the top and bottom vertices
should be identified, and that the labelling of the edges terminating 
at $r_1$ has been suppressed. Figure \ref{fig_ex_construction_K} shows
the left Krieger cover of $X$, but the structure of the irreducible
component corresponding to the left Fischer cover has been suppressed
to emphasize the structure of the higher layers. 
\end{example}

\begin{figure}
\begin{center}
\begin{tikzpicture}
  [bend angle=45,
   knude/.style = {circle, inner sep = 0pt},
   vertex/.style = {circle, draw, minimum size = 1 mm, inner sep =
      0pt, fill=black!10},
   textVertex/.style = {rectangle, draw, minimum size = 6 mm, inner sep =
      1pt, fill=black!10},
   to/.style = {->, shorten <= 1 pt, >=stealth', semithick}]
  
  \node[circle, draw, minimum size = 22 mm] (F) at
                           (0,0) {$(F, \LL_F)$}; 
  \node[textVertex] (vr) at (0, -0.65) {$r_1$};

  \node[textVertex] (x) at (-3,-3) {$P_\infty (a_x^\infty)$}; 
  \node[textVertex] (y) at (0,-3) {$P_\infty (a_y^\infty)$}; 
  \node[textVertex] (z) at (3,-3) {$P_\infty (a_z^\infty)$}; 

  \draw[->, shorten <= 5 pt, shorten >= 16 pt, >=stealth', semithick]
  ($(vr)+(-0.2,0)$) -> node[midway, above] {$a_{r,x}^1$ \,} ($(x)+(-0.5,0)$);
  \draw[->, shorten <= 12 pt, shorten >= 16 pt, >=stealth', semithick]
  ($(vr)+(0,0)$) -> node[midway, below] {\,$a_{r,x}^2$} ($(x)+(-0.3,0)$) ;

  \draw[->, shorten <= 10 pt, shorten >= 10 pt, >=stealth', semithick]
  ($(vr)+(-0.07,0)$) to node[auto,swap] {$a_{r,y}^1$} ($(y)+(-0.07,0)$);
  \draw[->, shorten <= 10 pt, shorten >= 10 pt, >=stealth', semithick]
  ($(vr)+(0.07,0)$) to node[auto] {$a_{r,y}^2$} ($(y)+(0.07,0)$);

  \draw[->, shorten <= 12 pt, shorten >= 15 pt, >=stealth', semithick]
  ($(vr)+(0.0,0.0)$) -> ($(z)+(0.0,0)$);
  \draw[->, shorten <=  5 pt, shorten >= 15 pt, >=stealth', semithick]
  ($(vr)+(0.2,0.0)$) -> ($(z)+(0.2,0)$);
  \draw[->, shorten <= -2 pt, shorten >= 15 pt, >=stealth', semithick]
  ($(vr)+(0.4,0.0)$) -> ($(z)+(0.4,0)$);
  \draw[->, shorten <= -9 pt, shorten >= 15 pt, >=stealth', semithick]
  ($(vr)+(0.6,0.0)$) -> 
  node[near start, right] {\, $a_{r,z}^1,a_{r,z}^2,a_{r,z}^3,a_{r,z}^4$} ($(z)+(0.6,0)$);

  \draw[->, shorten <= 23 pt, shorten >= 23 pt, >=stealth', semithick]
  ($(y)+(0,0.07)$) -> node[auto] {$a_{y,z}^1$} ($(z)+(0,0.07)$);
  \draw[->, shorten <= 23 pt, shorten >= 23 pt, >=stealth', semithick]
  ($(y)+(0,-0.07)$) -> node[auto,swap] {$a_{y,z}^2$} ($(z)+(0,-0.07)$);

  \draw[to, loop below] (x) to node[auto] {$a_x$} (x);
  \draw[to, loop below] (y) to node[auto] {$a_y$} (y);
  \draw[to, loop below] (z) to node[auto] {$a_z$} (z);

\end{tikzpicture}
\end{center}
\caption{Left Krieger cover of the shift space $X$ considered in
  Example \ref{ex_construction}. 
  The structure of the irreducible component corresponding to the left
  Fischer cover has been suppressed. 
}
\label{fig_ex_construction_K}
\end{figure}
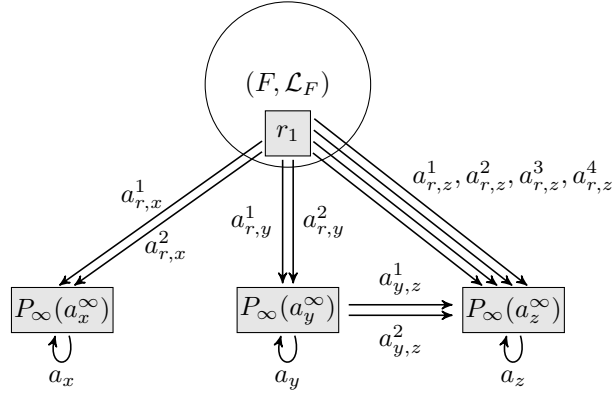  

In \cite{bates_eilers_pask} it was also remarked that an 
invariant analogous to the one discussed in Proposition
\ref{prop_range_invariant} is obtained by considering the proper
communication graph of the right Krieger  cover. The following example
shows that the two invariants may carry different information.

\begin{example}
\label{ex_2_invariants}
The labelled graph in Figure \ref{fig_2_invariants_1} is 
left-resolving, irreducible, and predecessor-separated, so it is the
left Fischer 
cover of an irreducible sofic shift. Similarly, the labelled graph in
Figure \ref{fig_2_invariants_2} is irreducible, right-resolving and
follower-separated, so it is the right Fischer cover of an irreducible
sofic shift. By considering the edges labelled $d$, it is easy to see
that the two graphs present the same sofic shift space $X$.  

Every right-ray which contains a letter different from $a$ or $a'$ is
intrinsically synchronizing, so consider a right-ray $x^+ \in 
X^+$ such that $(x^+)_i \in \{a, a'\}$ for all $i \in \N$.
By considering Figure \ref{fig_2_invariants_1}, it is clear that
$P_\infty(x^+) = P_\infty(u) \cup P_\infty(v) \cup 
P_\infty(y) = P_\infty(y)$, so $P(x^+)$ is also in the first layer of the left Krieger cover. Hence, the proper communication graph has only one vertex and no edges.

Every left-ray containing a letter different from $a$ or $a'$ is
intrinsically synchronizing, so consider the left-ray $a^\infty \in
X^-$. Figure \ref{fig_2_invariants_2} shows that $F_\infty(a^\infty)
= F_\infty(u') \cup F_\infty(v')$ and that no single 
vertex $y'$ in the right Fischer cover has $F_\infty(y') =
F_\infty(a^\infty)$, so there is a vertex in the second layer
of the right Krieger cover. In particular, the corresponding proper
communication graph is non-trivial. 
\end{example}

\begin{figure}[t!]
\begin{center}
\begin{tikzpicture}
  [bend angle=45,
   knude/.style = {circle, inner sep = 0pt},
   vertex/.style = {circle, draw, minimum size = 1 mm, inner sep =
      0pt, fill=black!10},
   textVertex/.style = {rectangle, draw, minimum size = 6 mm, inner sep =
      1pt, fill=black!10},
   to/.style = {->, shorten <= 1 pt, >=stealth', semithick}]
  
  \matrix[row sep=10mm, column sep=15mm]{
    & \node[textVertex] (u) {$u$}; &  \\
    \node[textVertex] (v) {$v$}; & \node[textVertex] (w) {$w$}; 
                             & \node[textVertex] (x) {$x$};  \\
    & \node[textVertex] (y) {$y$}; & \\
  };
  \draw[to, loop left]    (v) to node[auto] {$a'$} (v);
  \draw[to, loop above]   (u) to node[auto] {$a$} (u);
  \draw[to, loop below]   (y) to node[auto] {$a$} (y);
  \draw[to, loop above]   (v) to node[auto] {$a$} (v);
  \draw[to, loop left]    (u) to node[auto] {$a'$} (u);
  \draw[to, loop left]    (y) to node[auto] {$a'$} (y);
  \draw[to, bend right=20] (w) to node[auto,swap] {$b$} (v);
  \draw[to] (x) to node[auto,swap] {$f$} (u);
  \draw[to] (w) to node[auto,swap] {$g$} (x);
  \draw[to] (u) to node[auto,swap] {$e$} (w);
  \draw[to, bend right=20] (v) to node[auto,swap] {$c$} (w);
  \draw[to, bend left=20] (y) to node[auto] {$d$} (w);
  \draw[to] (x) to node[auto] {$f$} (y);
  \draw[to, bend left=20] (w) to node[auto] {$b$} (y);
\end{tikzpicture}
\end{center}
\caption{Left Fischer cover of the irreducible sofic shift $X$ discussed
  in Example \ref{ex_2_invariants}.} 
\label{fig_2_invariants_1}
\end{figure}
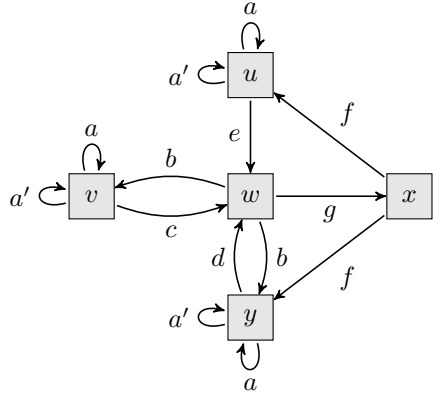
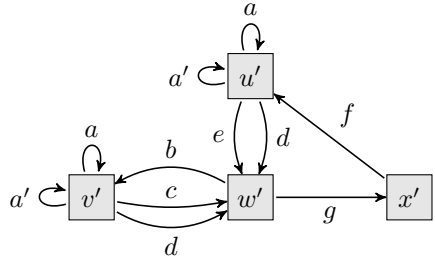
\begin{figure}[t!] 
\begin{center}
\begin{tikzpicture}
  [bend angle=45,
   knude/.style = {circle, inner sep = 0pt},
   vertex/.style = {circle, draw, minimum size = 1 mm, inner sep =
      0pt, fill=black!10},
   textVertex/.style = {rectangle, draw, minimum size = 6 mm, inner sep =
      1pt, fill=black!10},
   to/.style = {->, shorten <= 1 pt, >=stealth', semithick}]
  
  \matrix[row sep=10mm, column sep=15mm]{
    & \node[textVertex] (u') {$u'$}; &  \\
    \node[textVertex] (v') {$v'$}; & \node[textVertex] (w') {$w'$}; 
                             & \node[textVertex] (x') {$x'$};  \\
  };
  \draw[to, loop left]    (v') to node[auto] {$a'$} (v');
  \draw[to, loop above]   (u') to node[auto] {$a$} (u');
  \draw[to, loop above]   (v') to node[auto] {$a$} (v');
  \draw[to, loop left]    (u') to node[auto] {$a'$} (u');
  \draw[to, bend right=30] (w') to node[auto,swap] {$b$} (v');
  \draw[to] (x') to node[auto,swap] {$f$} (u');
  \draw[to] (w') to node[auto,swap] {$g$} (x');
  \draw[to, bend right=20] (u') to node[auto,swap] {$e$} (w');
  \draw[to, bend left=20] (u') to node[auto] {$d$} (w');
  \draw[to, bend right=30] (v') to node[auto,swap] {$d$} (w');
  \draw[to, bend right=10] (v') to node[auto] {$c$} (w');
\end{tikzpicture}
\end{center}
\caption{Right Fischer cover of the irreducible sofic shift $X$ discussed
  in Example \ref{ex_2_invariants}.}
\label{fig_2_invariants_2}
\end{figure}

\section{$\Cs$-Algebras associated to sofic shift spaces}
\label{sec_cs} 
Cuntz and Krieger \cite{cuntz_krieger} introduced a class of
$\Cs$-algebras which can naturally be viewed as the universal
$\Cs$-algebras associated to shifts of finite type.
This was generalized by Matsumoto \cite{matsumoto_1997} who associated
two $\Cs$-algebras $\OO_X$ and $\OO_{X^\ast}$ to every shift space
$X$, and these Matsumoto algebras have been studied intensely 
\cite{carlsen, katayama_matsumoto_watatani,
  matsumoto_1997, matsumoto_1998, 
  matsumoto_1999_dimension_group,matsumoto_1999_relations,
matsumoto_1999_simple,
  matsumoto_2000_automorphisms,matsumoto_2000_stabilized,
  matsumoto_2001, matsumoto_2002, matsumoto_watatani_yoshida}.
The two Matsumoto algebras $\OO_X$ and $\OO_{X^\ast}$ are generated
by elements satisfying the same relations, but they are not isomorphic
in general \cite{carlsen_matsumoto}. 
This paper will follow the approach of Carlsen in \cite{carlsen_2008}
where a universal $\Cs$-algebra $\OO_{\tilde X}$ is associated to every
one-sided shift space $\tilde X$. 
This 
also gives a way to associate $\Cs$-algebras to every two-sided
shift since a two-sided shift $X$ corresponds to two one-sided 
shifts $X^+$ and $X^-$.

\subsubsubsection{Ideal lattices}
Let $X$ be a sofic shift space and let $\OO_{X^+}$ be the universal
$\Cs$-algebra associated to the one-sided shift $X^+$ as defined in
\cite{carlsen_2008}. 
Carlsen proved that $\OO_{X^+}$ is isomorphic to the Cuntz-Krieger
algebra of the left Krieger cover of $X$ \cite{carlsen}, so the
lattice of gauge invariant ideals in $\OO_{X^+}$ is given by the
proper communication graph of the left Krieger cover of $X$
\cite{bates_pask_raeburn_szymanski,kumjian_pask_raeburn_renault}, and all ideals are given in this way if the left Krieger cover satisfies Condition (K) \cite[Theorem 4.9]{raeburn}.
Hence, Proposition \ref{prop_layers} and Theorem \ref{thm_gfc_foundation_lkc} can be used to investigate the ideal
lattice of $\OO_{X^+}$. For a reducible sofic shift, a part of the ideal lattice is given by the structure of the generalized left Fischer cover, which is reducible, but if $X$ is an irreducible sofic shift, and the left Krieger cover of
$X$ satisfies Condition (K) then the fact that the left Krieger cover
has a unique top component implies that $\OO_{X^+}$ will always have
a unique maximal ideal.
The following proposition shows that all these lattices can be
realized.   

\begin{prop}
\label{prop_ideal_lattice_of_irreducible_shift}
Any finite lattice of ideals with a unique maximal ideal is the ideal
lattice of the universal $\Cs$-algebra $\OO_{X^+}$ associated to an
AFT shift $X$.   
\end{prop}

\begin{proof}
Let $E$ be a finite directed graph whitout circuits and with a unique
maximal vertex.
Consider the following slight modification of the algorithm from the
proof of Proposition \ref{prop_range_invariant}. 
For each $v \in E$, draw two loops of length 1 at each vertex $v_1,
\ldots, v_{n(v)}$ associated to $v$: 
One labelled $a_v$ and one labelled $a_v'$. The rest of the
construction is as before.
Let $(K, \LL_K)$ be the left Krieger cover of the corresponding
sofic shift. As before, the proper communication graph of $K$ is given by $E$, and now $(K, \LL_K)$
satisfies Condition (K),
so there is a bijective correspondence between the hereditary
subsets of $E^0$ and the ideals of $\Cs(K) \cong \OO_{X^+}$.
Since $E$ was arbitrary, any finite ideal lattice with a unique maximal ideal can be obtained in this way.  
\end{proof}

\subsubsubsection{The $\Cs$-algebras $\OO_{X^+}$ and $\OO_{X^-}$}
Every two-sided shift space $X$ corresponds to two one-sided shift
spaces $X^+$ and $X^-$, and this gives two natural ways to associate a
universal $\Cs$-algebra to $X$.
The next goal is to show that these two $\Cs$-algebras may carry different information about the shift space.
Let $\OO_{X^-}$ be the universal $\Cs$-algebra associated to the
one-sided shift space $(X^\textrm{T})^+$ as defined in
\cite{carlsen_2008}. The left Krieger cover of $X^\textrm{T}$ is the
transpose of the right Krieger cover of $X$, so by \cite{carlsen},
$\OO_{X^-}$ is isomorphic to the Cuntz-Krieger algebra of the
transpose of the right Krieger cover of $X$.

\begin{example}
\label{ex_+-_not_isomorphic}
Let $X$ be the sofic shift from Example \ref{ex_2_invariants}. Note
that the left and right Krieger covers of $X$ both satisfy Condition
(K) from \cite{raeburn}, so the corresponding proper communication graphs
completely determine the ideal lattices of $\OO_{X^+}$ and $\OO_{X^-}$.   
The proper communication graph of the left Krieger cover $(K, \LL_K)$ of $X$ is
trivial, so $\OO_{X^+}$ is simple, while there are precisely two
vertices in the proper communication graph of the right Krieger
cover of $X$, so there is exactly one non-trivial ideal in $\OO_{X^-}$.
In particular, $\OO_{X^+}$ and $\OO_{X^-}$ are not isomorphic.

Consider the edge shift $Y = \X_{K}$. This is an SFT, and the left and
right Krieger covers of $Y$ are both $(K, \LL_{\Id})$,  where
$\LL_{\Id}$ is the identity map on the edge set $K^1$. By \cite{carlsen},
$\OO_{X^+}$ and  $\OO_{Y^+}$ are isomophic to $\Cs (K)$.
Similarly, $\OO_{Y^-}$ is isomorphic to $\Cs (K^\textrm T)$, and
$K^\textrm T$ is an irreducible graph satisfying Condition (K), so
$\OO_{Y^-}$ is simple.
In particular, $\OO_{Y^-}$ is not isomophic to $\OO_{X^-}$.  
This shows that the $\Cs$-algebras
associated to $X^+$ and $X^-$ are not always isomorphic, and that
there can exist a shift space $Y$ such that $\OO_{Y^+}$ is
isomorphic to $\OO_{X^+}$ while $\OO_{Y^-}$ is not isomorphic to
$\OO_{X^-}$. 
\end{example}

\subsubsubsection{An investigation of Condition ($\ast$)}
In \cite{carlsen_matsumoto}, two $\Cs$-algebras $\OO_X$ and
$\OO_{X^\ast}$ are associated to every two-sided shift space $X$.
The $\Cs$-algebras $\OO_{X}$, $\OO_{X^\ast}$, and $\OO_{X^+}$  
are generated by partial isometries satisfying the same
relations, but $\OO_{X^+}$ is always universal 
unlike $\OO_{X}$ \cite{carlsen_2008}.  
In \cite{carlsen_matsumoto}, it is proved that
$\OO_X$ and $\OO_{X^*}$ are isomorphic when $X$ satisifies a
condition called Condition ($\ast$). 
The example from \cite[Section 4]{carlsen_matsumoto} shows that not
all sofic shift spaces satisfy this condition by constructing a
sofic shift where the left Krieger cover and the past set cover are
not isomorphic. 
The relationship between Condition ($\ast$) and the structure of the left
Krieger cover and the past set cover
is further clarified by the final main result.
For each $l \in \N$ and $w \in \BB(X)$ define 
$P_l(w) = \{ v \in \BB(X) \mid vw \in \BB(X), |v| \leq l \}$. Two
words $v, w \in \BB(X)$ are said to be $l$-past equivalent if $P_l(v)
= P_l(w)$. For $x^+ \in X^+$, $P_l(x^+)$ and $l$-past equivalence are
defined analogously.  

\begin{cond}[$\ast$]
For every $l \in \N$ and every infinite $F \subseteq \BB(X)$ such
that $P_l(u) = P_l(v)$ for all $u,v \in F$ there exists $x^+ \in X^+$ such
that $P_l(w) = P_l(x^+)$ for all $w \in F$.
\end{cond}

\begin{lem}
\label{lem_infinitely_many_w}
A vertex $P$ in the past set cover of a sofic shift $X$ is in an essential subgraph
if and only if there exist infinitely many $w \in \BB(X)$ such that
$P_\infty(w) = P$. 
\end{lem}

\begin{proof}
Let $P$ be a vertex in an essential subgraph of the past set cover of
$X$, and let $x^+ \in  X$ be a right ray with a presentation starting at
$P$. Given $n \in \N$, there exists $w_n \in \BB(X)$ such that
$P = P_\infty(x_1 x_2 \ldots x_n w_n)$. To prove the converse, 
let $P$ be a vertex in the past set cover for which there exist infinitely many $w \in \BB(X)$ such that
$P = P_\infty(w)$. For each $w$, there is a path labelled $w_{[1,\rvert w \lvert -1]}$ starting at $P$. There are no sources in the past set cover, so this implies that $P$ is not stranded.
\end{proof}

\begin{prop}
\label{prop_condition_star}
A sofic shift $X$ satisfies Condition ($\ast$) if and only if the
left Krieger cover is the maximal essential subgraph of the past set
cover. 
\end{prop}

\begin{proof}
Assume that $X$ satisfies Condition ($\ast$). Let $P$ be a vertex in an
essential subgraph of the past set cover and define $F = \{ w \in
\BB(X) \mid P_\infty(w) = P \}$. 
Choose $m \in \N$ such that for all $x,y
\in \BB(X) \cup X^+$, $P_\infty(x) = P_\infty(y)$ if and only if
$P_m(x) = P_m(y)$. 
By Lemma \ref{lem_infinitely_many_w}, $F$ is an infinite set, so
Condition ($\ast$) can be used to choose $x^+ \in X^+$ such that
$P_m(x^+) = P_m(w)$ for all $w \in F$. By the choice of $m$, this means
that $P_\infty(x^+) =  P_\infty(w) = P$ for all $w \in
F$, so $P$ is a vertex in the left Krieger cover.  

To prove the other implication, assume that the left Krieger cover is
the maximal essential subgraph of the past set cover. Let $l \in \N$
be given, and consider an infinite set $F \subseteq \BB(X)$ for which
$P_l(u) = P_l(v)$ for all $u,v \in F$. Since $X$ is sofic, there are
only finitely many different predecessor sets, so there must exist $w
\in F$ such that $P_\infty(w) = P_\infty(v)$ for infinitely many $v
\in F$. By Lemma \ref{lem_infinitely_many_w}, this proves that $P =
P_\infty(w)$ is a vertex in the maximal essential subgraph of the past
set cover. By assumption, this means that it is a vertex in the left
Krieger cover, so there exists $x^+ \in X^+$ such that $P_\infty(w) =
P_\infty(x^+)$. In particular, $P_l(x^+) = P_l(w) = P_l(v)$ for all $v
\in F$, so Condition ($\ast$) is satisfied.
\end{proof}

\noindent
In \cite{bates_pask} it was proved that $\OO_{X^\ast}$ is isomorphic
to the Cuntz-Krieger algebra of the past set cover of $X$ when $X$
satisfies a condition called Condition (I).
According to Carlsen \cite{carlsen_privat}, a proof similar to the
proof which shows that $\OO_{X^+}$ is isomorphic to the Cuntz-Krieger
algebra of the left Krieger cover of $X$
should prove
that $\OO_{X^\ast}$ is isomorphic to the
Cuntz-Krieger algebra of the subgraph of the past set cover of $X$
induced by the vertices $P$ for which there exist
infinitely many words $w$ such that $P_\infty(w) = P$.
Using Lemma \ref{lem_infinitely_many_w}, this shows that
$\OO_{X^\ast}$ is always isomorphic to the Cuntz-Krieger algebra of
the maximal essential subgraph of the past set cover of $X$.

\bibliographystyle{abbrv}

\Addresses
\end{document}